\documentclass[12pt,a4paper]{article}
\usepackage[utf8]{inputenc}
\usepackage{amsmath}
\usepackage{amsfonts}
\usepackage{amssymb}
\usepackage{graphicx}
\usepackage{mathabx}

    \usepackage{pdfsync}

\newtheorem{proposition}{Proposition}

\newenvironment{proof}[1][Proof]{\textbf{#1.} }{\ \rule{0.5em}{0.5em}}
\author{Yves Le Jan}
\title{Homology of Brownian loops } 
\begin{document}
\maketitle

The purpose of this note is to extend to Brownian loops some homology and holonomy results obtained in the case of discrete loops on a graph (see \cite{lejanito},
\cite{lejanbakry}, and \cite{stfl}).

\bigskip

To understand our construction, note first the following: In the finite graph case  with exponential holding times considered in \cite{stfl}, the infinitesimal generator of the semigroup $P_t$ is $I-P$, , with $P$ a submarkovian matrix. If $-\rho_i$ denote the eigenvalues of the infinitesimal generator $P-I$, $$- \ln(\det (I-P)=\sum- \ln(\rho_i)=\zeta'(0)$$
where $\zeta(s)=\sum \rho_i^{-s}=Tr((I-P)^{-s}=\frac{1}{\Gamma(s)}\mu(T^{s})$
is the Mellin transform of $Tr(P_t)$:
$$\zeta(s)=\frac{1}{\Gamma(s)}\int_0^{\infty}t^{s-1}Tr(P_t)dt.$$\\
With the notations of \cite{stfl}, \cite{lejanito}, and 
\cite{lejanbakry}, we have, for any positive $s$, the following results which yield the distribution of the vertex and edge occupation fields under the regularized loop measure $T^s (l)\mu(dl)$:
 $$\mu(T^{s}(e^{-\left\langle \widehat{l},\chi\right\rangle }
-1))=\Gamma(s)[Tr((I-\frac{\lambda}{\lambda+\chi}P)^{-s})-Tr((I-P)^{-s})]$$
and 
$$\mu(T^{s}(e^{-\left\langle N(l),Z\right\rangle }
-1))=\Gamma(s)[Tr((I-P\ast {Z})^{-s})-Tr((I-P)^{-s})].$$
 These expressions converge respectively, as $s$ decreases to zero, towards $\ln(\frac{\det(I-\frac{\lambda}{\lambda+\chi})}{\det(I-P)}) $ and $\ln(\frac{\det(I-P\ast {Z})^{-s})}{\det(I-P)})$\\
  which yield the distribution of the vertex and edge occupation fields under the loop measure $\mu(dl)$.

\bigskip
We consider in the following a compact Riemannian manifold $X$ with metric tensor $g_{i,j}$ and continuous positive killing rate $k$. Lebesgue measure is denoted by $dx$. The energy form is defined on smooth functions by:
\[
e(f,f)=\frac{1}{2}\int g^{i,j}(x)\frac{\partial f}{\partial x_{i}%
}\frac{\partial f}{\partial x_{j}}\det(g)^{-\frac{1}{2}}(x)dx+\int k(x)f(x)^{2}%
\det(g)^{-\frac{1}{2}}(x)dx.
\]
The associated heat semigroup is denoted  $P_t$. The corresponding infinitesimal generator is
\[
A=\frac{1}{2}\Delta_{x}-k(x).
\]

\medskip
We denote by $\mu$ the loop measure associated with the corresponding continuous $\det(g)^{-\frac{1}{2}}(x)dx$-symmetric Markov process, i.e. the killed Brownian motion on $X$. It can also be viewed as a shift invariant measure on based loops. We can refer to \cite{LW}, \cite{stfl} and \cite{rflj} for the general definition in terms of Markovian bridges,
The Poisson process of Brownian loops (Lawler and Werner's ''loop soup'') is
defined in the same way as in the graph case.\\

\medskip

For the Brownian motion on a $d$-dimensional compact Riemannian manifold killed at positive rate, the same integral $\frac{1}{\Gamma(s)}\int_0^{\infty}t^{s-1}Tr(P_t)dt$ defines 
$\zeta(s)$ for $s>\frac{d}{2}$. It can be extended into a meromorphic function on $\mathbb{C}$ with one pole at $\dfrac{d}{2}$.\\
This follows from M-P asymptotics of the heat kernel (see \cite{MP})  and from well known properties of the Mellin transform (see for example \cite{Zag}).
The expression $e^{-\zeta'(0)}$ defines the zeta-regularized determinant of $-A$, denoted $\det'(-A)$. Such determinants have already been used in the context of loop measures on Riemann surfaces (see \cite{Dub}). \\
In the case of the Brownian motion on a $d$-dimensional torus killed at constant rate $m^2$, $\zeta(s)=\sum_{k_i \in  \mathbb{N}^d} (m^2+\frac{1}{2} \sum k_i^2)^{-2s}$ for $s\in ]\frac{d}{2},\infty[$.

The Poisson process of loops of intensity $\alpha \mu$ associated with the Brownian motion on a $d$-dimensional compact Riemannian manifold killed at positive rate is  denoted $\mathcal{L}_{\alpha}$. It defines a random walk 
on the homology group  of $X$, denoted $H_1(\mathbb{Z})$ obtained by adding the contributions of each loop. We will determine its distribution via a Fourier integral on the Jacobian torus, $ Jac=H^1(\mathbb{R})/ H^1(\mathbb{Z})$. \\Here, we denote by $H^1(\mathbb{R})$ the space of harmonic one-forms and by $H^1(\mathbb{Z})$ the space of harmonic one-forms $\omega$ such that for all loops $\gamma$ the holonomy $\omega(\gamma)$ is an integer.\\
 We need to determine, for any harmonic one-form $\omega\in H^1(\mathbb{R})$, the integral $\int (e^{2\pi i\int_l\omega}-1)\mu(dl)$. Note that:\\
 
a) The definition of the integral of the closed $1$-form on a non-smooth loop causes no difficulties: $\int_l \omega=\int_{l'}\omega$ if $l'$ is a smooth loop  close enough to $l$ (with $T(l)=T(l')$). Indeed, a smooth loop which is not homotopic to zero has a minimum positive diameter and if the uniform distance between two smooth loops $l'$ and $l"$ is small enough, we can see by cutting them into path segments of small diameter and joining the extremities of these segments by geodesics that $\int_{l'}\omega=\int_{l"}\omega$.\\

b) The integral $\int (e^{2\pi i\int_l\omega}-1)\mu(dl)$ is finite, as short loops have zero homology.\\

c) $\int T(l)^s (e^{2\pi i\int_l\omega}-1)\mu(dl)$ is holomorphic in $s$. \\

d)  $\zeta_{\omega}(s)= \dfrac{1}{\Gamma(s)}\int_0^{\infty}t^{s-1}Tr(P_t^{\omega})dt.$ is well defined on $s>\frac{n}{2}$ with by definition $P_t^{\omega}(x,y)=\int e^{2\pi i\int_{l}\omega} \mathbb{P}_t^{x,y}(dl)$.\\

$\zeta_{\omega}$ is the zeta function associated with the generator $A+2\pi i\langle \omega,d \rangle)- 2\pi^2 \Vert \omega \Vert^2$.\\
From c), $\zeta_{\omega}(s)$ can be extended into a meromorphic function on $\mathbb{C}$ with one pole at $\frac{n}{2}$. Then we easily get the following:
 \begin{proposition} 
$E(e^{2\pi i \sum_{l\in \mathcal{L}_{\alpha} }\int_l \omega})=e^{\alpha [ \zeta_{\omega}'(0)-\zeta'(0)]}=\left[ \frac{\det'(-A)}{\det'(-A-i\langle \omega,d \rangle)+2\pi^2 \Vert \omega \Vert^2}\right ]^{\alpha }.$
 \end{proposition}
 \begin{proof}  

$\int (e^{2\pi i\int_l\omega}-1)\mu(dl)=\dfrac{d}{ds}_{ | s=0}\frac{1}{\Gamma(s)}\int T^s( e^{2\pi i\int_l\omega}-1)\mu(dl)= \zeta_{\omega}'(0)-\zeta'(0)$, as the reciprocal gamma function vanishes and has unit derivative in zero.\\
\end{proof} 
 
 \medskip
 
 The distribution of the induced random homology $h^{(\alpha)}$ defined by the identity $\langle h^{(\alpha)} ,\omega, \rangle= \sum_{l\in \mathcal{L}_{\alpha} }\int_l \omega$ for any $\omega\in H^1(\mathbb{R}) $ can be computed as a Fourier integral on the Jacobian torus.  Let $d\omega$ be the Lebesgue measure on $ Jac$ and let $\vert Jac\vert$ denote its volume. Then:


 \begin{proposition} 
  For all $j \in H_1(\mathbb{Z})$, we have:
 
$$ P(h^{(\alpha)}=j)=\frac{1}{\vert Jac\vert}\int_{Jac} \left[ \frac{\det'(-A)}{\det'(-A-2\pi i\langle \omega,d \rangle)+ 2\pi^2 \Vert \omega \Vert^2}\right ]^{\alpha }e^{-2\pi i\langle  j,\omega \rangle} \:d\omega.$$
 \end{proposition}
 \begin{proof}  
 $$P(h^{(\alpha)}=j)=\frac{1}{\vert Jac\vert}\int_{Jac}E(e^{2\pi i \langle h^{(\alpha)} ,\omega, \rangle}  e^{-2\pi i\langle  j,\omega \rangle})\:d\omega$$
$$ =\frac{1}{\vert Jac\vert}\int_{Jac}E(e^{2\pi i \sum_{l\in \mathcal{L}_{\alpha} }\int_l \omega}  e^{-2 \pi i\langle  j,\omega \rangle})\:d\omega
=\frac{1}{\vert Jac\vert}\int_{Jac}e^{\alpha [ \zeta_{\omega}'(0)-\zeta'(0)]}  e^{-2\pi i\langle  j,\omega \rangle}\:d\omega.$$
\end{proof} 

%
%
Analogous results can be given for the point occupation field in dimension one.\\

\bigskip

\textbf{ Generalisation: }Given a compact gauge group $G$ and an element $\gamma$ of the moduli space of flat connections on the trivial bundle $X \times G$, let us denote $H_{\gamma}(l)$ the conjugacy class of $G$ obtained by taking the holonomy of a loop $l$ from any base point. Given any finite dimensional unitary representation $\pi$ of $G$, let $Tr_{\pi}(H_{\gamma}(l))$ be the trace of its image by the representation.\\
 Let $\zeta_{\gamma,\pi}(s)$ be  the meromorphic extension of $ \dfrac{1}{\Gamma(s)}\int_0^{\infty}t^{s-1}Tr(P_t^{\gamma,\pi})dt$, 
 well defined on $s>\frac{d}{2}$ with for any $u,v \in \mathbb{C}^{dim(\pi)}$ $$P_t^{\gamma,\pi}((x,u),(y,v))=\int [\pi(H_{\gamma}(l))]_{u,v} \mathbb{P}_t^{x,y}(dl).$$\\
 As in the first proposition, an expression can be given for the expectation of the product of the traces of the loop holonomies.\\
 $$E( \prod_{l\in \mathcal{L}_{\alpha} }Tr_{\pi}(H_{\gamma}(l)) )=e^{\alpha [ \zeta_{\gamma,\pi}'(0)-\zeta'(0)]}.$$
 For related results in the finite graph case, see  \cite{lejanito} and \cite{lejanbakry}. See also \cite{LK}, and \cite{Ken} for bundles on graphs.\\  
 Holonomy classes cannot be added as homology classes do. Still some composition results can be given. See \cite{lejanbakry},  section  4-3 for the finite graph case.


\bigskip

\noindent

  D\'epartement de Math\'ematique. Universit\'e Paris-Sud.  Orsay, France.

\bigskip
   yves.lejan@math.u-psud.fr

\end{document}